\documentclass[12pt,draft,reqno]{amsart}
\usepackage{amsmath}
\usepackage{amssymb}
\usepackage{amsthm}
\usepackage{bm}
\numberwithin{equation}{section}
\newtheorem{theorem}{Theorem}[section]
\newtheorem{lemma}[theorem]{Lemma}
\newtheorem{proposition}[theorem]{Proposition}

\theoremstyle{definition}
\newtheorem*{remark}{Remark}

\allowdisplaybreaks
\begin{document}
\title[On zeros of bilateral zeta functions]{On zeros of bilateral Hurwitz and periodic zeta and zeta star functions}
\author[T.~Nakamura]{Takashi Nakamura}
\address[T.~Nakamura]{Department of Liberal Arts, Faculty of Science and Technology, Tokyo University of Science, 2641 Yamazaki, Noda-shi, Chiba-ken, 278-8510, Japan}
\email{nakamuratakashi@rs.tus.ac.jp}
\urladdr{https://sites.google.com/site/takashinakamurazeta/}
\maketitle

\begin{abstract}
In this paper, we show the following; \\
(1) The periodic zeta function ${\rm{Li}}_s (e^{2\pi ia})$ with $0<a<1/2$ or $1/2 < a <1$ does not vanish on the real line. \\
(2) All real zeros of $Y(s,a):=\zeta (s,a) - \zeta (s,1-a)$, $O(s,a) := -i {\rm{Li}}_s (e^{2\pi ia}) + i{\rm{Li}}_s (e^{2\pi i(1-a)})$ and $X(s,a) := Y(s,a) + O(s,a)$ with $0 < a < 1/2$ are simple and only at the negative odd integers. \\
(3) All real zeros of $Z(s,a):=\zeta (s,a) + \zeta (s,1-a)$ are simple and only at the non-positive even integers if and only if $1/4 \le a \le 1/2$. \\
(4) All real zeros of $P(s,a):={\rm{Li}}_s (e^{2\pi ia}) + {\rm{Li}}_s (e^{2\pi i(1-a)})$ are simple and only at the negative even integers if and only if $1/4 \le a \le 1/2$. 

Moreover, the asymptotic behavior of real zeros of $Z(s,a)$ and $P(s,a)$ are studied when $0 < a < 1/4$. In addition, the complex zeros of these zeta functions are also discussed when $0 <a <1/2$ is rational or transcendental. 
\end{abstract}

\section{Introduction and statement of main results}

\subsection{Real zeros of the Hurwitz and periodic zeta functions}
Let $s = \sigma +it$, where $\sigma,t \in {\mathbb{R}}$. Then the Riemann zeta function is defined by
$$
\zeta (s) := \sum_{n=1}^\infty \frac{1}{n^s}, \qquad \sigma >1.
$$
It is well known that $\zeta (s)$ is continued meromorphically and has a simple pole at $s=1$ with residue $1$. The Riemann zeta function $\zeta (s)$ satisfies the functional equation 
\begin{equation}\label{eq:RZfe}
\zeta(1-s) = \frac{2\Gamma (s)}{(2\pi )^s} \cos \Bigl( \frac{\pi s}{2} \Bigr) \zeta(s)
\end{equation}
(see for instance \cite[(2.1.8)]{Tit}). Moreover, all real zeros of $\zeta(s)$ are simple and only at the negative even integers. The Riemann hypothesis is a conjecture that the Riemann zeta function $\zeta(s)$ has its non-real zeros only on the so-called  critical line $\sigma =1/2$. 

Next we define the Hurwitz zeta function $\zeta (s,a)$ by 
$$
\zeta (s,a) := \sum_{n=0}^\infty \frac{1}{(n+a)^s}, \qquad \sigma >1, \quad 0<a \le 1.
$$
The Hurwitz zeta function $\zeta (s,a)$ is a meromorphic function with a simple pole at $s=1$ whose residue is $1$ (see for example \cite[Chapter 12]{Apo}). Note that all real zeros of $\zeta (s,1/2)$ are simple and only at the non-positive even integers. 

Define the periodic zeta function ${\rm{Li}}_s (e^{2\pi ia})$ by  
$$
{\rm{Li}}_s (e^{2\pi ia}) := \sum_{n=1}^\infty \frac{e^{2\pi ina}}{n^s}, \qquad \sigma >1, \quad 0<a \le 1
$$
(see for instance \cite[Exercise 12.2]{Apo}). It should be mentioned that ${\rm{Li}}_s (e^{2\pi ia})$ with $0<a<1$ is analytically continuable to the whole complex plane since the Dirichlet series of ${\rm{Li}}_s (e^{2\pi ia})$ converges uniformly in each compact subset of the half-plane $\sigma >0$ when $0<a<1$  (see for instance \cite[p.~20]{LauGa}). 

For real zeros of $\zeta (s,a)$ and ${\rm{Li}}_s (z)$ with $|z| \le 1$, we can refer to \cite[Section 1.1]{NaC}. Recently, Matsusaka \cite[Theorem 1.3]{Mts} showed that for integers $N \le -1$, the Hurwitz zeta function $\zeta (s,a)$ has real zeros in the interval $(-N-1,-N)$ if and only if $B_{N+1}(a) B_{N+2}(a)$ $<0$, where $B_N(a)$ is the $N$-th Bernoulli polynomial. It should be noted that the cases $N=-1$ and $N=0$ are proved by the author in \cite[Theorem 1.2]{NaC} and \cite[Theorem 1.2]{NJMAA}, respectively. 

By numerical calculation, we can see that the difference between two successive real zeros of $\zeta (s,a)$ is not $2$ in general, namely, the difference depends on $0<a<1$. On the contrary, it is well-known that the gap between consecutive real zeros of $\zeta (s)$ is $2$ by the functional equation (\ref{eq:RZfe}) and the fact that $\zeta (s)$ does not vanish on the real line if $\sigma \ge 0$. Moreover, the difference between two successive real zeros of Dirichlet $L$-function $L(s,\chi)$ is $2$ if we assume the non-existence of any real zero of $L(s,\chi)$ in the interval $(0,1)$. As an unconditional result, Conrey and Soundararajan \cite{ConSou} proved that a positive proportion of real primitive Dirichlet characters $\chi$, the associated Dirichlet $L$-function $L(s,\chi)$ does not vanish on the positive real axis. 

\subsection{Main results}
For $0 <a \le 1/2$, let
\begin{equation*}
\begin{split}
&Z(s,a) := \zeta (s,a) + \zeta (s,1-a), \qquad P(s,a) := {\rm{Li}}_s (e^{2\pi ia}) + {\rm{Li}}_s (e^{2\pi i(1-a)})\\
&Y(s,a) := \zeta (s,a) - \zeta (s,1-a), \qquad 
O(s,a):= -i \bigl( {\rm{Li}}_s (e^{2\pi ia}) - {\rm{Li}}_s (e^{2\pi i(1-a)}) \bigr) ,\\
&X(s,a) := Y(s,a) + O(s,a) = 
\zeta (s,a) - \zeta (s,1-a) -i \bigl( {\rm{Li}}_s (e^{2\pi ia}) - {\rm{Li}}_s (e^{2\pi i(1-a)}) \bigr).
\end{split}
\end{equation*}
We name $Z(s,a)$ bilateral Hurwitz zeta function and call $Y(s,a)$ bilateral Hurwitz zeta star function.  In addition, we call $P(s,a)$ bilateral periodic zeta function and name $O(s,a)$ bilateral periodic zeta star function.  Without the name of these function, they have already appeared in an old paper Fine \cite[Equations (9), (11), (12) and (15)]{Fine} and some famous text books such as Karatsuba $\&$ Voronin \cite[Equations (9) and (12) in \S1.4]{KaVo} and Koblitz \cite[Problems 3 (c) and 5 (d) in \S 2.4]{Kob}. Furthermore, it is showed in \cite[Section 5]{NaIn} that $Z(s,|a|)$ appears as the spectral density of some stationary self-similar Gaussian distributions. Note that the functional equation and zeros on the critical line of $Z(s,a)+P(s,a)$ have been already discussed in \cite[Theorems 1.1 and 1.2]{NPFE}.

In the present paper, we prove the following four main theorems which describe real zeros of the zeta-functions ${\rm{Li}}_s (e^{2\pi ia})$, $Y(s,a)$, $O(s,a)$, $X(s,a)$ $Z(s,a)$ and $P(s,a)$. It should be emphasised that the gap between consecutive real zeros of $Y(s,a)$, $O(s,a)$, $X(s,a)$ with $0 < a < 1/2$, and  $Z(s,a)$ and $P(s,a)$ with $1/4 \le a \le 1/2$ is always $2$, namely, the the gaps do not depend on $a$ just like the Riemann zeta function.

\begin{theorem}\label{th:li1}
When $0 < a < 1/2$ or $1/2 < a < 1$, the periodic zeta function ${\rm{Li}}_s (e^{2\pi ia})$ does not vanish on the real line.
\end{theorem}

The theorem above asserts the non-existence of any real zero of ${\rm{Li}}_s (e^{2\pi ia})$. On the contrary, the theorem below can be regarded as an analogue of the well-know fact that the all real zeros of any Dirichlet $L$-function $L(s,\chi)$ with $\Re(s) < 0$ are only at the negative odd integers if $\chi$ is a primitive character with $\chi(-1) = -1$ (see Section 1.1). 

\begin{theorem}\label{th:OY1}
All real zeros of the functions $Y(s,a)$, $O(s,a)$ and $X(s,a)$ with $0 < a < 1/2$ are simple and only at the negative odd integers. 
\end{theorem}

As analogues of the real zeros of the Riemann zeta function $\zeta (s)$ or Dirichlet $L$-functions $L(s,\chi)$ with primitive characters satisfying the condition $\chi(-1) = 1$, we have the following two statements. 

\begin{theorem}\label{th:m1}
All real zeros of the function $Z(s,a)$ are simple and only at the non-positive even integers if and only if $1/4 \le a \le 1/2$.
\end{theorem}
\begin{theorem}\label{th:mt1}
All real zeros of the function $P(s,a)$ are simple and only at the negative even integers if and only if $1/4 \le a \le 1/2$.
\end{theorem}

Next we consider the zeta function $Z(s,a)$ and $P(s,a)$ with $0< a <1/4$. The asymptotic behavior of real zeros $\beta_Z (a)$ and $\beta_P (a)$ below are given in (3) of Propositions \ref{pro:c1}, \ref{pro:a1} and \ref{pro:a2} (see also Propositions \ref{cor:Z} and \ref{cor:p}). 
\begin{proposition}\label{pro:z1}
Assume $0 < a < 1/4$. Then the function $Z(s,a)$ have real zeros at the non-positive integers and another real zero $\beta_Z (a)$ in the open interval $(-\infty, 1)$ which does not appear in Theorem \ref{th:m1}. 
\end{proposition}
\begin{proposition}\label{pro:p1}
Suppose $0 < a < 1/4$. Then the function $P(s,a)$ have precisely one simple real zero $\beta_P (a)$ in $(0, \infty)$. Furthermore, all real zeros in $(-\infty,0)$ of the function $P(s,a)$ are simple and only at the negative even integers. 
\end{proposition}

At the end of this subsection, we discuss the complex or non-real zeros of the zeta functions $Y(s,a)$, $O(s,a)$, $X(s,a)$ $Z(s,a)$ and $P(s,a)$. 
For $k=2,3,4$, let  $\chi_{-k}$ be the real and non-principal Dirichlet character mod $k$. Then we define the Dirichlet $L$-function by $L(s,\chi_{-k}) := \sum_{n=1}^\infty \chi_{-k} (n) n^{-s}$ for $k=3,4,6$ and $\Re (s) >1$. These functions can be continued to holomorphic functions on the whole complex plane by analytic continuation. For $k=3,4$, the Generalized Riemann hypothesis for $L(s,\chi_{-k})$ are conjectures that $L(s,\chi_{-k})$ have their non-real zeros only on the critical line $\sigma =1/2$. Furthermore, the Generalized Riemann hypothesis for $L(s,\chi_{-6})$ states that all non-real zeros of $L(s,\chi_{-6})$ are only on the vertical lines $\sigma = 1/2$ and $\sigma = 0$ (see Section 3.1). For $a= 1/6$, $1/4$, $1/3$ and $1/2$, we have the following. 

\begin{proposition}\label{pro:1}
Suppose that $a= 1/6$, $1/4$, $1/3$ or $1/2$. Then the Riemann hypothesis is true if and only if\\
$\bullet \,$ all non-real zeros of $Z(s,a)$ are on the vertical lines $\sigma = 1/2$ and $\sigma = 0$, or\\
$\bullet \,$ all non-real zeros of $P(s,a)$ are on the vertical lines $\sigma = 1/2$ and $\sigma = 1$.
\end{proposition}

\begin{proposition}\label{pro:yox1}
Assume that $a= 1/6$, $1/4$ or $1/3$.\\
$\bullet \,$ For each $k=3,4$, all non-real zeros of $Y(s,1/k)$ and $O(s,1/k)$ are on the vertical line $\sigma = 1/2$ if and only if the Generalized Riemann hypothesis for $L(s,\chi_{-k})$ is true. \\
$\bullet \,$ All non-real zeros of $Y(s,1/6)$ and $O(1-s,1/6)$ are on the vertical lines $\sigma = 1/2$ and $\sigma = 0$ if and only if the Generalized Riemann hypothesis for $L(s,\chi_{-6})$ is true. \\
$\bullet \,$ For each $k=3,4,6$, all non-real zeros of $X(s,1/k)$ are on the vertical line $\sigma = 1/2$ \\if and only if the Generalized Riemann hypothesis for $L(s,\chi_{-k})$ is true. 
\end{proposition}

On the contrary, when $a \ne 1/2, 1/3,/ 1/4, 1/6$, we have the following for non-periodic complex zeros of the functions ${\rm{Li}}_s (e^{2\pi ia})$, $O(s,a)$, $Y(s,a)$ and $X(s,a)$, $Z(s,a)$ and $P(s,a)$.
\begin{proposition}\label{pro:zero1}
Let $a \in {\mathbb{Q}} \cap (0,1/2) \setminus \{1/6, 1/4, 1/3\}$. Then for any $\delta >0$, there exist positive constants $C_a^\flat (\delta)$ and $C_a^\sharp (\delta)$ such that the function $P(s,a)$ has more than $C_a^\flat (\delta) T$ and less than $C_a^\sharp (\delta)T$ complex zeros in the rectangle $1 < \sigma < 1+\delta$ and $0 < t < T$ if $T$ is sufficiently large. Moreover, for any $1/2 < \sigma_1 < \sigma_2 <1$, there are positive numbers $C_a^\flat (\sigma_1,\sigma_2)$ and $C_a^\sharp (\sigma_1,\sigma_2)$ such that the function $P(s,a)$ has more than $C_a^\flat (\sigma_1,\sigma_2) T$ and less than $C_a^\sharp (\sigma_1,\sigma_2) T$ non-trivial zeros in the rectangle $\sigma_1 < \sigma < \sigma_2$ and $0 < t < T$ when $T$ is sufficiently large. 

Furthermore, the function ${\rm{Li}}_s (e^{2\pi ia})$ with $a \in {\mathbb{Q}} \cap (0,1/2)$, the functions $O(s,a)$ and $X(s,a)$ with $a \in {\mathbb{Q}} \cap (0,1/2) \setminus \{1/6, 1/4, 1/3\}$, and the functions $Z(s,a)$ and $Y(s,a)$ with $a \in (0,1/2) \setminus \overline{\mathbb{Q}}$ or $a \in {\mathbb{Q}} \cap (0,1/2) \setminus \{1/6, 1/4, 1/3\}$ have the property mentioned above. 
\end{proposition}

This paper is organized as follows. 
In Section 2, we give some remarks on linear relations between Dirichlet $L$-functions and zeta functions composed by $\zeta (s,a)$ and ${\rm{Li}}_s (e^{2\pi ia})$ and their real zeros. 
We give proofs of Proposition \ref{pro:yox1}, Theorems \ref{th:li1} and \ref{th:OY1} in Section 3. We prove Theorems \ref{th:m1} and \ref{th:mt1}, Propositions \ref{pro:z1}, \ref{pro:p1}, \ref{pro:1} and \ref{pro:zero1} in Section 4. 

\section{Linear relations and remarks on real zeros}
\subsection{Linear relations}
First, we show that each of the functions $Z(s,r/q)$, $P(s,r/q)$, $Y(s,r/q)$ and $O(s,r/q)$ essentially can be expressed as a linear combination of Dirichlet $L$-functions if $0 < r < q$ are relatively prime integers.
\begin{proposition}
Let $0 < r < q$ be relatively prime integers, $\varphi $ be the Euler totient function. Put $\overline{\chi}_\pm (r) := (1 \pm \chi(-1)) \overline{\chi} (r)$ and $G_r^\pm (\chi) := (1 \pm \chi(-1)) \sum_{n=1}^q e^{2\pi irn/q} \chi (n)$. Then one has
\begin{equation}\label{eq:qq}
\begin{split}
&Z(s,r/q) = \frac{q^s}{\varphi (q)} \! \sum_{\chi \!\!\!\! \mod q} \!
\overline{\chi}_+ (r)  L(s,\chi), \qquad
P(s,r/q) = \frac{1}{\varphi (q)} \! \sum_{\chi \!\!\!\! \mod q} \! G_r^+(\overline{\chi})  L(s,\chi).
\end{split}
\end{equation}
Furthermore, for $0 < 2r < q$, it holds that
\begin{equation}\label{eq:yoq}
\begin{split}
Y(s,r/q) = \frac{q^s}{\varphi (q)} \! \sum_{\chi \!\!\!\! \mod q} \! \overline{\chi}_- (r)  L(s,\chi), \qquad
O(s,r/q) = \frac{-i}{\varphi (q)} \! \sum_{\chi \!\!\!\! \mod q} \! G_r^- (\overline{\chi})  L(s,\chi).
\end{split}
\end{equation}
\end{proposition}

\begin{proof}
When $0 < r < q$ are relatively prime integers, we have 
\begin{equation}\label{eq:k1}
\begin{split}
&\zeta (s,r/q) = \sum_{n=0}^\infty \frac{1}{(n+r/q)^s} = \sum_{n=0}^\infty \frac{q^s}{(r+qn)^s} =
\frac{q^s}{\varphi (q)} \sum_{\chi \!\!\!\! \mod q} \overline{\chi} (r) L(s,\chi), \\
&{\rm{Li}}_s  (e^{2\pi i r/q}) = q^{-s} \sum_{n=1}^q e^{2\pi irn/q} \zeta (s,n/q) =
\frac{1}{\varphi (q)} \sum_{\chi \!\!\!\! \mod q} G(r,\overline{\chi}) L(s,\chi).
\end{split}
\end{equation}
Note that $\varphi (q) \le 2$ if and only if $q=1,2,3,4,6$, which plays important role in the proof of Proposition \ref{pro:zero1}. By (\ref{eq:k1}) and the definitions of the functions $Z(s,r/q)$, $P(s,r/q)$, $Y(s,r/q)$ and $O(s,r/q)$, we have (\ref{eq:qq}) and (\ref{eq:yoq}). 
\end{proof}

Next we show that the Dirichlet $L$-function essentially can be written by a linear combination of the functions $Z(s,r/q)$, $P(s,r/q)$, $Y(s,r/q)$ or $O(s,r/q)$. 
\begin{proposition}
Let $\chi$ be a primitive Dirichlet character, $G(\chi)$ be the Gauss sum associated to $\chi$, and $0 < r < q$ be relatively prime integers. If $\chi(-1) =1$, we have
\begin{equation}\label{eq:qqgyakuZP}
L (s,\chi) =  \frac{1}{2q^ s} \sum_{r=1}^q \chi (r) Z(s,r/q) = \frac{1}{2G(\overline{\chi})} \sum_{r=1}^q \overline{\chi} (r) P(s,r/q).
\end{equation}
When $\chi(-1) =-1$, one has
\begin{equation}\label{eq:qqgyakuYO}
L (s,\chi) =  \frac{1}{2q^ s} \sum_{r=1}^q \chi (r) Y(s,r/q) = \frac{i}{2G(\overline{\chi})} \sum_{r=1}^q \overline{\chi} (r) O(s,r/q).
\end{equation}
\end{proposition}

\begin{proof}
When $0 < r < q$ are relatively prime integers, it holds that
\begin{equation}\label{eq:qqgyaku}
\begin{split}
&L (s,\chi) = \sum_{r=1}^q \sum_{n=0}^{\infty} \frac{\chi (r+nq)}{(r+nq)^s} = \sum_{r=1}^q \chi (r) \sum_{n=0}^{\infty} \frac{1}{(r+nq)^s} = q^{-s} \sum_{r=1}^q \chi (r) \zeta (s,r/q) ,\\
&L (s,\chi) = \frac{1}{G(\overline{\chi})} \sum_{n=1}^{\infty} \sum_{r=1}^q  \frac{\overline{\chi} (r) e^{2\pi i rn/q}}{n^s} 
= \frac{1}{G(\overline{\chi})} \sum_{r=1}^q \overline{\chi} (r) {\rm{Li}}_s (e^{2\pi i r/q}).
\end{split}
\end{equation}
Suppose that $\chi$ is even, namely, $\chi(-1) =1$. Then we have
\begin{equation*}
\begin{split}
2q^ s L (s,\chi) = \sum_{r=1}^q \chi (r) \zeta (s,r/q) + \sum_{r=1}^q \chi (q-r) \zeta (s,1-r/q) 
= \sum_{r=1}^q \chi (r) Z(s,r/q) 
\end{split}
\end{equation*}
from (\ref{eq:qqgyaku}) and $\chi (q-r) = \chi (-r) = \chi (r)$. Similarly, one has
\begin{equation*}
\begin{split}
&2G(\overline{\chi})L (s,\chi) = 
\sum_{r=1}^q \overline{\chi} (r) {\rm{Li}}_s (e^{2\pi i r/q}) + \sum_{r=1}^q \overline{\chi} (q-r) {\rm{Li}}_s (e^{2\pi i (q-r)/q}) 
= \sum_{r=1}^q \overline{\chi} (r) P(s,r/q).
\end{split}
\end{equation*}
Similar arguments apply to the case $\chi(-1) = -1$. Thus, we have (\ref{eq:qqgyakuZP}) and (\ref{eq:qqgyakuYO}). 
\end{proof}

\subsection{Remarks on real zeros}
We have the following remarks by theorems in Section 1.2, the equations (\ref{eq:qq}), (\ref{eq:yoq}), (\ref{eq:qqgyakuZP}) and (\ref{eq:qqgyakuYO}).
\begin{remark}
By Theorem \ref{th:OY1}, the functions $Y(s,a)$ and $O(s,a)$ do not vanish in the open set ${\mathfrak{O}} := \bigcup_{n \in {\mathbb{Z}}} (n,n+1)$. Therefore, form (\ref{eq:qqgyakuYO}), it is natural to expect that $L(s,\chi)$ with $\chi(-1) = -1$ does not vanish in ${\mathfrak{O}}$, especially in the open interval $(0,1)$, since all the functions $Y(s,a)$ and $O(s,a)$ with $0 < a < 1/2$, and $L(s,\chi)$ with odd Dirichlet characters do not have zeros in ${\mathfrak{O}} \setminus (0,1) = \bigcup_{0 \ne n \in {\mathbb{Z}}} (n,n+1)$ and all the functions $Y(s,a)$ and $O(s,a)$ with $0 < a < 1/2$ do not vanish in the interval $(0,1)$. 
\end{remark}

\begin{remark}
On the contrary, the functions $Z(s,a)$ and $P(s,a)$ have real zeros in the open set ${\mathfrak{O}}$ if and only if $0 < a < 1/6$ or $1/6 < a < 1/4$ from Theorems \ref{th:m1}, \ref{th:mt1} and \begin{equation}\label{eq:Z0P1w1/6}
Z(0, 1/6) = P (1,1/6)=0
\end{equation}
which is proved by (\ref{eq:z1/6}) and (\ref{eq:p1/6}). Recall that by (\ref{eq:qqgyakuZP}), the Dirichlet $L$-function $L(s,\chi)$ with $\chi(-1) = 1$ can be expressed as a linear combination of the functions $Z(s,r/q)$ and $P(s,r/q)$, where $0 < r < q$ are relatively prime integers. In addition, the functions $Z(s,r/q)$ and $P(s,r/q)$ with $0 < r/q < 1/6$ have a real zero in the interval $(0,1)$ from (2) of Proposition \ref{pro:c1}. Despite of these facts, it is expected that any Dirichlet $L$-function $L(s,\chi)$ with $\chi(-1) = 1$ does not vanish in the interval $(0,1)$. 
\end{remark}

\begin{remark}
There is a possibility that the two remarks above indicate that proving the non-existence of real zeros in $(0,1)$ for Dirichlet $L$-functions with even characters is more difficult than that with odd characters.
\end{remark}

\section{Proofs of the main results of ${\rm{Li}}_s (e^{2\pi ia})$, $Y(s,a)$ and $O(s,a)$}
\subsection{Proof of Proposition \ref{pro:yox1}}
For each $k=3,4,6$, we show that $Y(s,1/k)$, $O(s,1/k)$ and $X(s,1/k)$ can be essentially expressed as $L(s,\chi_{-k})$. 
\begin{proof}[Proof of Proposition \ref{pro:yox1}]
Let $a=1/3$ and $\Re(s) >1$. Then one has
\begin{equation*}
\begin{split}
Y(s,1/3) = 3^s \sum_{n=0}^\infty \frac{1}{(3n+1)^s} - 3^s \sum_{n=0}^\infty \frac{1}{(3n+2)^s} = 3^s L(s,\chi_{-3}). 
\end{split}
\end{equation*}
In addition, it holds that
$$
O(s,1/3) = \sum_{n=1}^\infty \frac{e^{2 \pi i n/3}}{i n^s} - \sum_{n=1}^\infty \frac{e^{-2 \pi i n/3}}{i n^s} =
\sum_{n=0}^\infty \frac{\sqrt{3}}{(3n+1)^{s}} - \sum_{n=0}^\infty \frac{\sqrt{3}}{(3n+2)^{s}} 
= \sqrt{3} L(s,\chi_{-3}). 
$$
Therefore, we have
$$
X(s,1/3) = Y(s,1/3) + O(s,1/3) = \bigl( 3^s + \sqrt{3} \bigr)L(s,\chi_{-3}).
$$

Similarly, we can show that
\[
Y(s,1/4) = 4^s L(s,\chi_{-4}), \qquad O(s,1/4) = 2 L(s,\chi_{-4}), \qquad X(s,1/4) = \bigl( 4^s + 2 \bigr)L(s,\chi_{-4}). 
\]

Finally, consider the case $a=1/6$. For $\Re(s) >1$, we have
\begin{equation*}
\begin{split}
Y(s,1/6) &= \sum_{n=0}^\infty \frac{1}{(n+1/6)^s} - \sum_{n=0}^\infty \frac{1}{(n+5/6)^s} =
6^s \sum_{n=0}^\infty \frac{1}{(6n+1)^s} - 6^s \sum_{n=0}^\infty \frac{1}{(6n+5)^s} \\ 
&= 6^s L(s,\chi_{-6}) = 6^s \bigl( 1 + 2^{-s} \bigr) L(s,\chi_{-3}) = \bigl( 6^s + 3^s \bigr) L(s,\chi_{-3}) 
\end{split}
\end{equation*}
(see for instance \cite[Lemma 10.2.1]{Cohen}). On the other hand, one has
\begin{equation*}
\begin{split}
O(s,1/6) & =
\sum_{n=0}^\infty \frac{\sqrt{3}}{(6n+1)^{s}} + \sum_{n=0}^\infty \frac{\sqrt{3}}{(6n+2)^{s}} 
- \sum_{n=0}^\infty \frac{\sqrt{3}}{(6n+4)^{s}} - \sum_{n=0}^\infty \frac{\sqrt{3}}{(6n+5)^{s}} \\ 
&= \sqrt{3} L(s,\chi_{-6}) + 2^{-s} \sqrt{3} L(s,\chi_{-3}) = \sqrt{3} \bigl( 1 + 2^{1-s} \bigr) L(s,\chi_{-3}). 
\end{split}
\end{equation*}
Therefore, one has
$$
X(s,1/6) = Y(s,1/6) + O(s,1/6) = \bigl( 3^s ( 1 + 2^s) + \sqrt{3} ( 1 + 2^{1-s}) \bigr)L(s,\chi_{-3}).
$$
Now we show the factor $3^s ( 1 + 2^s) + \sqrt{3} ( 1 + 2^{1-s})$ has zeros only on $\sigma =1/2$. Put
$$
f(s) := \frac{3^s}{\sqrt{3}}, \qquad g(s) := \frac{ 1 + 2^{1-s}}{1 + 2^s}.
$$
Obviously, one has $|f(s)| = 1$ for $\sigma = 1/2$, $|f(s)| > 1$ if $\sigma > 1/2$ and $|f(s)| < 1$ when $\sigma < 1/2$, and $|g(s)| = 1$ for $\sigma = 1/2$. Now we show that $|g(s)| < 1$ when $\sigma > 1/2$ and $|g(s)| > 1$ when $\sigma < 1/2$. Suppose $\sigma > 1/2$ and $\cos (t\log 2) \ge 0$. Then, by
$$
g(\sigma + it) = 
\frac{1 +2^{1-\sigma} \cos (t\log 2) - i 2^{1-\sigma} \sin (t\log 2)}{1 +2^\sigma \cos (t\log 2) + i 2^\sigma \sin (t\log 2)},
$$
we have $|g(s)| < 1$. Next assume $\sigma > 1/2$ and $-1 \le \cos (t\log 2) < 0$. In this case, we only have to show the inequality 
$$
1 + 2^{2-\sigma} \cos (t\log 2) + 2^{2-2\sigma} < 1 + 2^{1+\sigma} \cos (t\log 2) + 2^{2\sigma}
$$
which is equivalent to $( 2^{2-\sigma} -2^{1+\sigma}) \cos (t\log 2) < 2^{2\sigma} - 2^{2-2\sigma}$.
From the factorization
\[
2^{2\sigma} - 2^{2-2\sigma} - \bigl( 2^{1+\sigma} - 2^{2-\sigma} \bigr) =
2^{-2\sigma} ( 2^{2\sigma} -2) \bigl( (2^{\sigma} -1)^2 +1 \bigl),
\]
for $\sigma >1/2$ and $-1 \le \cos (t\log 2) < 0$, we get
$$
\bigl( 2^{2-\sigma} -2^{1+\sigma} \bigr) \cos (t\log 2) \le 2^{1+\sigma} - 2^{2-\sigma} < 2^{2\sigma} - 2^{2-2\sigma}.
$$
Therefore, we have $|g(s)| < 1$ when $\sigma > 1/2$. By using $g(1-s) = 1/g(s)$, we immediately obtain $|g(s)| > 1$ when $\sigma < 1/2$. 
\end{proof}

\subsection{Proof of Theorem \ref{th:li1}}
Let $0<a<1$. For $\sigma >0$, it is known that 
\begin{equation}\label{eq:LiIn}
{\rm{Li}}_s (e^{2\pi ia}) = \sum_{n=1}^\infty \frac{e^{2\pi ina}}{n^s} = 
\frac{e^{2\pi ia}}{\Gamma (s)} \int_0^\infty \frac{x^{s-1}}{e^x-e^{2\pi ia}} dx
\end{equation}
(see for example \cite[(2.9)]{NaC}). It should be noted that the series $\sum_{n=1}^\infty n^{-s} e^{2\pi ina}$ with $0<a <1$ converges uniformly on compact subsets in the half-plane $\sigma >0$ by Abel's summation formula (see for instance \cite[p.~20]{LauGa}). From the following functional equation
\begin{equation}\label{eq:feli}
{\rm{Li}}_s (e^{2\pi ia}) = \frac{\Gamma (1-s)}{(2\pi )^{1-s}} 
\Bigl( e^{\pi i(1-s)/2} \zeta (1-s,a) + e^{-\pi i(1-s)/2} \zeta (1-s,1-a) \Bigr),
\end{equation}
we can extend the definition of ${\rm{Li}}_s(e^{2\pi ia})$ over the entire $s$-plane when $0<a<1$  (see for instance \cite[Exercises 12.2 and 12.3]{Apo}). By an easy computation, we have 
\begin{equation}\label{eq:LiInIm}
\begin{split}
\Im \biggl( \frac{e^{2\pi ia}}{e^x-e^{2\pi ia}} \biggr) = 
\frac{e^x \sin 2\pi a}{(e^x - \cos 2\pi a)^2 + \sin ^2 2\pi a} .
\end{split}
\end{equation}

\begin{proof}[Proof of Theorem \ref{th:li1}]
First suppose $0 < a < 1/2$. From (\ref{eq:LiIn}) and (\ref{eq:LiInIm}), it holds that
\begin{equation}\label{eq:nvLi}
\Im \bigl( {\rm{Li}}_\sigma (e^{2\pi ia}) \bigr) > 0, \qquad \sigma >0, \quad 0 < a < 1/2.
\end{equation}

Next we show ${\rm{Li}}_0 (e^{2\pi ia}) \ne 0$. It is widely known that we have
\begin{equation}\label{eq:1res}
\zeta (s,a) = \frac{a^{1-s}}{s-1} + f(s,a),
\end{equation}
where $f(s,a)$ is an analytic function for all $s \in {\mathbb{C}}$ (see for instance \cite[Theorem 12.21]{Apo}). According to L'Hospital's rule, (\ref{eq:feli}) and the formula above, one has
\begin{equation*}
\begin{split}
{\rm{Li}}_0 (e^{2\pi ia}) = \lim_{s \to 0}
\frac{\Gamma (1-s)}{(2\pi )^{1-s}} \frac{a^s e^{\pi i(1-s)/2} + (1-a)^s e^{-\pi i(1-s)/2}}{-s} 
 = - \frac{1}{2} + \frac{i}{2\pi} \log (a^{-1}-1) \ne 0.
\end{split}
\end{equation*}

For $\sigma >1$, it holds that
\begin{equation}\label{eq:parDa}
\frac{\partial}{\partial a} \zeta (s,a) = \sum_{n=0}^\infty \frac{\partial}{\partial a} (n+a)^{-s}
= -s \sum_{n=0}^\infty (n+a)^{-s-1} = -s \zeta (s+1,a). 
\end{equation}
Moreover, we have $\zeta (\sigma,a) >0 $ if $\sigma >1$ by the series expression of $\zeta (s,a)$. Hence one has
\begin{equation}\label{eq:zineq1}
\zeta (\sigma,a) > \zeta (\sigma, 1-a) > 0, \qquad \sigma > 1, \quad 0 < a < 1/2.
\end{equation}
Hence, the equation 
$$
e^{\pi i(1-\sigma)/2} \zeta (1-\sigma,a) + e^{-\pi i(1-\sigma)/2} \zeta (1-\sigma,1-a) = 0
$$ 
contradicts to the facts that $|e^{\pi i(1-\sigma)}|=1$ and
$$
|\zeta (1-\sigma,a)| > |\zeta (1-\sigma,1-a)|, \qquad \sigma <0, \quad 0 < a < 1/2
$$ 
which is proved by (\ref{eq:zineq1}). Thus we have ${\rm{Li}}_\sigma (e^{2\pi ia}) \ne 0$ when $\sigma <0$  and $0 < a < 1/2$ by the functional equation (\ref{eq:feli}) and the fact $\Gamma (1-\sigma) > 0$ if $\sigma <0$. 

We can prove ${\rm{Li}}_\sigma (e^{2\pi ia}) \ne 0$ when $1/2 < a < 1$ similarly since one has
\begin{equation*}
\begin{split}
&\Im \bigl( {\rm{Li}}_\sigma (e^{2\pi ia}) \bigr) < 0, \qquad \sigma >0, \quad 1/2 < a < 1,\\
&\zeta (\sigma,1-a) > \zeta (\sigma, a) > 0, \qquad \sigma > 1, \quad 1/2 < a < 1
\end{split}
\end{equation*}
from (\ref{eq:nvLi}) and (\ref{eq:zineq1}), respectively. 
\end{proof}

\subsection{Proof of Theorem \ref{th:OY1}}
We can see that the function $Y(s,a)$ is entire when $0<a<1/2$ by using (\ref{eq:1res}). Furthermore, the function $O(s,a)$ is entire if $0<a<1/2$ since ${\rm{Li}}_s(e^{2\pi ia})$ with $0<a<1/2$ is an entire function (see Section 3.2). 

For simplicity, we put
\[
\Gamma_{\! \pi} (s) : = \frac{\Gamma (s)}{(2\pi )^s}, \qquad
\Gamma_{\!\! \rm{cos}} (s) : = 2 \Gamma_{\! \pi} (s) \cos \Bigl( \frac{\pi s}{2} \Bigr),  \qquad 
\Gamma_{\!\! \rm{sin}} (s) : = 2 \Gamma_{\! \pi} (s) \sin \Bigl( \frac{\pi s}{2} \Bigr) .
\]
The following functional equation is well-known (see for instance \cite[Theorem 2.3.1]{LauGa})
\begin{equation}\label{eq:laugafe1}
\zeta (1-s,a) =  \Gamma_{\!\! \rm{cos}} (s) \sum_{n=1}^\infty \frac{\cos 2\pi na}{n^s} + \Gamma_{\!\! \rm{sin}} (s) \sum_{n=1}^\infty \frac{\sin 2\pi na}{n^s} , \qquad \sigma >1.
\end{equation}
Note that the equation above holds for $\sigma >0$ when $0 < a <1$. From (\ref{eq:laugafe1}), we have
\begin{equation}\label{eq:feY1}
\begin{split}
Y(1-s,a)  = 
2 \Gamma_{\!\! \rm{sin}} (s) \sum_{n=1}^\infty \frac{\sin 2\pi na}{n^s} 
= \Gamma_{\!\! \rm{sin}} (s) O(s,a).
\end{split}
\end{equation}
On the other hand, by the functional equation
\begin{equation}\label{eq:laugafe1ks}
{\rm{Li}}_{1-s} (e^{2\pi ia}) = 
\Gamma_{\! \pi} (s)  \bigl( e^{\pi is/2} \zeta (s,a) + e^{-\pi is/2} \zeta (s,1-a) \bigr),
\quad 0<a<1,
\end{equation}
(see for example \cite[Exercises 12.2]{Apo}), it holds that
\begin{equation}\label{eq:feO1}
O(1-s,a) =
\Gamma_{\!\! \rm{sin}} (s) \bigl( \zeta (s,a) - \zeta (s,1-a) \bigr) = 
\Gamma_{\!\! \rm{sin}} (s) Y(s,a) .
\end{equation}

\begin{proof}[Proof of Theorem \ref{th:OY1} for $Y(s,a)$]
For $\sigma >1$, we have
\begin{equation*}
\begin{split}
Y(s,a) =
a^{-s} - (1-a)^{-s} + \sum_{n=1}^\infty n^{-s} 
\biggl( \Bigl( 1+\frac{a}{n} \Bigr)^{-s} - \Bigl( 1+\frac{1-a}{n} \Bigr)^{-s} \biggr).
\end{split}
\end{equation*}
The last series converges absolutely when $\sigma >0$ by
$$
\Bigl( 1+\frac{a}{n} \Bigr)^{-s} - \Bigl( 1+\frac{1-a}{n} \Bigr)^{-s} =  
\frac{s}{n} + \frac{s(s+1)}{2} \frac{a^2 - (1-a)^2}{n^2} +\cdots . 
$$
Therefore, we obtain
\begin{equation}\label{eq:n0Y}
Y(\sigma, a) > 0 , \qquad \sigma >0, \quad 0<a<1/2 
\end{equation}
from the inequality $(n+a)^{-\sigma} > (n+1-a)^{-\sigma}$ if $0<a<1/2$. On the other hand, one has $\Im ( {\rm{Li}}_\sigma (e^{2\pi i(1-a)})) < 0$ from (\ref{eq:nvLi}). Thus we obtain
\begin{equation}\label{eq:nOO}
O(\sigma,a) > 0, \qquad \sigma >0, \quad 0 < a < 1/2.
\end{equation}
Thus, by (\ref{eq:feY1}), all real zeros of $Y(1-s,a)$ with $\sigma > 0$ and $0 < a < 1/2$ is caused by $\sin( \pi s /2 ) =0$ with $\sigma >0$.Therefore, all real zeros of the function $Y(s,a)$ with $\sigma \le 0$ are simple and only at $s=1-2n$, where $n \in {\mathbb{N}}$. 
\end{proof}

\begin{proof}[Proof of Theorem \ref{th:OY1} for $O(s,a)$]
By (\ref{eq:nOO}), we only have to show the case $\sigma \le 0$. From (\ref{eq:feO1}) and (\ref{eq:n0Y}), all real zeros of $O(1-s,a)$ with $\sigma > 0$ and $0 < a < 1/2$ is deduced by $\sin( \pi s /2 ) =0$.
Hence, the function $O(\sigma,a)$ with $\sigma \le 0$ vanishes only at $s=1-2n$, where $n \in {\mathbb{N}}$. 
\end{proof}

\begin{proof}[Proof of Theorem \ref{th:OY1} for $X(s,a)$]
From the functional equations (\ref{eq:feY1}) and (\ref{eq:feO1}), one has
\begin{equation}\label{eq:feX1}
X(1-s,a) = \Gamma_{\!\! \rm{sin}} (s)  X(s,a) .
\end{equation}
On the other hand, the inequalities (\ref{eq:n0Y}) and (\ref{eq:nOO}) imply
$$
X(s,a) := Y(s,a) + O(s,a) > 0, \qquad \sigma >0, \quad 0<a<1/2 .
$$
Therefore, all real zeros of the function $Y(s,a)$ are simple and only at $s=1-2n$, where $n \in {\mathbb{N}}$ by the functional equation (\ref{eq:feX1}). 
\end{proof}

\section{Proofs of the main results of $Z(s,a)$ and $P(s,a)$}

\subsection{Proof of Proposition \ref{pro:1}}
We prove that $Z(s,a)$ and $P(s,a)$ can be essentially expressed as $\zeta(s)$ for each $a= 1/2, 1/3, 1/4, 1/6$. 
\begin{proof}[Proof of Proposition \ref{pro:1}]
When $\sigma >1$, one has
\[
\sum_{n=1}^\infty \frac{1}{n^s} = \sum_{l=1}^m \sum_{n=0}^\infty \frac{1}{(mn+l)^s} = \frac{1}{m^s} \sum_{l=1}^m \zeta (s,l/m) .
\]
By using the equation above, we have
\begin{equation}\label{eq:z24}
Z(s,1/2) = 2 \zeta (s,1/2) = 2 (2^s-1) \zeta (s),
\end{equation}
\begin{equation}\label{eq:z13}
Z(s,1/3) = 3^s \zeta (s) - \zeta (s,1)
= (3^s-1) \zeta (s) ,
\end{equation}
\begin{equation}\label{eq:z14}
\begin{split}
Z(s,1/4) = 4^s \zeta (s) - \zeta (s,1/2) - \zeta (s,1) =  2^s (2^s-1) \zeta (s) ,
\end{split}
\end{equation}
\begin{equation}\label{eq:z1/6}
\begin{split}
Z(s,1/6) & = 6^s \zeta (s) - \zeta (s,1/3) - \zeta (s,1/2) - \zeta (s,2/3) - \zeta (s)\\
&= (2^s-1) (3^s-1) \zeta (s).
\end{split}
\end{equation}
The equations above imply Proposition \ref{pro:1} for $Z(s,1/6)$ with $a= 1/2, 1/3, 1/4, 1/6$. 

Obviously, for $\sigma >1$, we have
$$
\sum_{n=1}^\infty \frac{m}{(mn)^s} = 
\sum_{l=1}^m \sum_{n=1}^\infty \frac{e^{2\pi i ln/m}}{n^s} = \sum_{l=1}^m {\rm{Li}}_s (e^{2\pi i l/m}).
$$
From this equality, one has
\begin{equation}\label{eq:z24t}
P(s,1/2) = 2 {\rm{Li}}_s (e^{\pi i}) = 2 (2^{1-s}-1)\zeta (s),
\end{equation}
\begin{equation}\label{eq:z13t}
P(s,1/3) = 3^{1-s} \zeta (s) - \zeta (s) = (3^{1-s}-1) \zeta (s) .
\end{equation}
\begin{equation}\label{eq:z14t}
P(s,1/4) =
4^{1-s} \zeta (s) - {\rm{Li}}_s (e^{\pi i}) - \zeta (s) =  2^{1-s} (2^{1-s}-1) \zeta (s) ,
\end{equation}
\begin{equation}\label{eq:p1/6}
\begin{split}
P(s,1/6) &=
6^{1-s} \zeta (s) - {\rm{Li}}_s (e^{2\pi i/3}) - {\rm{Li}}_s (e^{\pi i}) - {\rm{Li}}_s (e^{4\pi i/3}) - \zeta (s) \\
&= (2^{1-s}-1) (3^{1-s}-1) \zeta (s).
\end{split}
\end{equation}
Hence we obtain Proposition \ref{pro:1} for $P(s,a)$ with $a= 1/2, 1/3, 1/4, 1/6$. 
\end{proof}

\subsection{Lemmas}
First, we show the following functional equations (see also \cite[the proof of Theorem 1.1]{NPFE}). 
\begin{lemma}\label{lem:1}
Assume $0 < a \le 1/2$. The function $Z(s,a)$ has the functional equation 
\begin{equation}\label{eq:zfe1}
Z (1-s,a) = \Gamma_{\!\! \rm{cos}} (s) P(s,a).
\end{equation}
Furthermore, the function $P(s,a)$ satisfies the function equation
\begin{equation}\label{eq:zfe1p}
P(1-s,a) = \Gamma_{\!\! \rm{cos}} (s) Z(s,a).
\end{equation}
\end{lemma}
\begin{proof}
From (\ref{eq:laugafe1}), it holds that
$$
Z (1-s,a) =  
2 \Gamma_{\!\! \rm{cos}} (s) \sum_{n=1}^\infty \frac{\cos 2\pi na}{n^s} 
= \Gamma_{\!\! \rm{cos}} (s) P(s,a)
$$
which proves the functional equation (\ref{eq:zfe1}). We remark that the equation above holds for not only $\sigma >1$ but also $\sigma >0$ if $0<a<1$ by applying Abel's summation formula to the Dirichlet series of $P(s,a)$. On the other hand, by (\ref{eq:laugafe1ks}) we have
\[
P(1-s,a) = \Gamma_{\!\! \rm{cos}} (s) \zeta (s,a) + \Gamma_{\!\! \rm{cos}} (s) \zeta (s,1-a) =\Gamma_{\!\! \rm{cos}} (s)  Z(s,a)
\]
which shows the functional equation (\ref{eq:zfe1p}).
\end{proof}

Next consider the values of $\lim_{s \to 1} Z(s,a)$, $Z(0,a)$, $P(1,a)$ and $P(0,a)$.
\begin{lemma}\label{lem:2}
Let $0 < a \le 1/2$. Then we have
\begin{equation}\label{eq:zv1}
\lim_{s \to 1+0} Z(s,a) = -\lim_{s \to 1-0} Z(s,a) = \infty, \quad Z (0,a) = 0,
\end{equation}
\begin{equation}\label{eq:pv1}
P (1,a) = -2 \log (2 \sin \pi a)  , \qquad P(0,a)= -1.
\end{equation}
\end{lemma}
\begin{proof}
From (\ref{eq:1res}), it holds that
$$
\lim_{s \to 1+0} Z(s,a) = \infty , \quad \lim_{s \to 1-0} Z(s,a) = -\infty.
$$
Furthermore, by (\ref{eq:1res}) and (\ref{eq:zfe1p}), one has
$$
P(0,a) = \lim_{s \to 1} 
\frac{2\Gamma (s)}{(2\pi )^s} \frac{a^{1-s}+(1-a)^{1-s}}{s-1} \cos \Bigl( \frac{\pi s}{2} \Bigr) 
=\frac{4}{2\pi} \frac{-\pi}{2} = -1 .
$$
When $0 < \theta < 2\pi$, it is well-known that
$$
\sum_{n=1}^\infty \frac{e^{i n\theta}}{n} =
- \log \Bigl( 2\sin \frac{\theta}{2} \Bigr) + i \Bigl( \frac{\pi}{2} - \frac{\theta}{2} \Bigr).
$$
Hence we have $P (1,a) = -2 \log (2 \sin \pi a)$. Moreover, one has
$$
Z(0,a) = \lim_{s \to 1} \frac{2\Gamma (s)}{(2\pi )^s} \cos \Bigl( \frac{\pi s}{2} \Bigr) P(s,a) = 0 
$$
from the functional equation (\ref{eq:zfe1}). 
\end{proof}

We prove the following by modifying the proof of \cite[Lemma 2.3]{ES}.
\begin{lemma}\label{lem:3}
Let $0 < a \le 1/4$. Then the function
$$
\bigl(-\log (\cos 2 \pi a) \bigr)^{-\sigma} \Gamma (\sigma) P(\sigma,a)
$$
is strictly increasing for $\sigma >0$. 
\end{lemma}
\begin{proof}
For $0<a <1$ and $\sigma >0$, one has
\begin{equation}\label{eq:intzc1}
\sum_{n=1}^\infty \frac{\cos 2\pi na}{n^\sigma} =
\Re \biggl( \frac{e^{2\pi ia}}{\Gamma (\sigma)} \int_0^\infty \frac{x^{\sigma-1}}{e^x-e^{2\pi ia}} dx \biggr)
\end{equation}
from (\ref{eq:LiIn}). We can easily see that
\begin{equation}\label{eq:PImk1}
\begin{split}
&\Re \biggl( \frac{e^{2\pi ia}}{e^x-e^{2\pi ia}} \biggr) =
\frac{e^x \cos 2\pi a -1}{(e^x - \cos 2\pi a)^2 + \sin ^2 2\pi a} .
\end{split}
\end{equation}
Define $G(a,x)$ by the right hand side of the equation above. Then we can find
\begin{equation*}
\begin{split}
&e^x \cos 2\pi a -1 < 0 \quad \mbox{when} \quad 0 < x < -\log (\cos 2 \pi a), \\
&e^x \cos 2\pi a -1 \ge 0 \quad \mbox{when} \quad x \ge -\log (\cos 2 \pi a),
\end{split}
\end{equation*}
if $0 < a \le 1/4$. On the other hand, the function given in this lemma can be rewritten by
$$
\alpha^{-\sigma} \Gamma (\sigma) P(\sigma,a) = 
\int_0^\alpha G(a,x) \Bigl( \frac{x}{\alpha} \Bigr)^\sigma \frac{dx}{x} +
\int_\alpha^\infty G(a,x) \Bigl( \frac{x}{\alpha} \Bigr)^\sigma \frac{dx}{x},
$$
where $\alpha := -\log (\cos 2 \pi a)>0$. The first integral is strictly increasing in $\sigma >0$ since one has $G(a,x) < 0$ and $(x / \alpha)^\sigma$ is strictly decreasing in $\sigma$ when $0 < x < \alpha$. Similarly, the second integral is also strictly increasing in $\sigma >0$ since one has $G(a,x) > 0$ and $(x / \alpha)^\sigma$ is increasing in $\sigma$ when $x \ge \alpha$. Therefore, the function $\alpha^{-\sigma} \Gamma (\sigma) P(\sigma,a)$ is strictly increasing for $\sigma >0$ when $0 < a \le 1/4$. 
\end{proof}

Finally, we show the functions $Z(\sigma, a)$ and $P(\sigma, a)$ are strictly decreasing with respect to $0 < a \le 1/2$ for fixed $0 <\sigma \ne 1$.
\begin{lemma}\label{lem:4}
Let $0 < a < 1/2$ and $\sigma >0$. Then one has
$$
\frac{\partial}{\partial a} Z (\sigma,a) < 0 \quad (\sigma \ne 1), \qquad 
\frac{\partial}{\partial a} P (\sigma,a) < 0.
$$
\end{lemma}
\begin{proof}
From (\ref{eq:parDa}) it holds that
\begin{equation*}
\begin{split}
\frac{\partial}{\partial a} Z (s,a)= 
s \sum_{n=0}^\infty \Bigl( (n+1-a)^{-s-1} - (n+a)^{-s-1} \Bigr) = -s Y(s+1,a)
\end{split}
\end{equation*}
for $s \ne 1$ and $\sigma >0$. The equation above and (\ref{eq:n0Y}) imply Lemma \ref{lem:4} for $Z (\sigma,a)$.

Suppose $\sigma > 2$. Then it holds that 
$$
\frac{\partial}{\partial a} P(s,a) = 
2\sum_{n=1}^\infty \frac{\partial}{\partial a}  \frac{\cos 2\pi na}{n^s} =
- 4 \pi \sum_{n=1}^\infty  \frac{\sin 2\pi na}{n^{s-1}} = -2\pi O (s-1,a) .
$$
The formula above can be continued to the whole complex plane ${\mathbb{C}}$. On the other hand, by (\ref{eq:n0Y}), (\ref{eq:nOO}), Theorem \ref{th:OY1} and the continuity of the functions $Y(\sigma,a)$ and $O(\sigma,a)$ with respect to $\sigma \in {\mathbb{R}}$, it holds that
$$
Y(\sigma,a) > 0 , \quad O(\sigma,a) > 0, \qquad  \sigma > -1.  
$$
Hence, we have $(\partial / \partial a) P (\sigma,a) < 0$ for $\sigma > 0$. 
\end{proof}

\subsection{When ${\bm{ \sigma \in (0,1)}}$}
The second and third statement of the following proposition are analogues of (1) and (2) in \cite[Theorem 1.2]{ES}, respectively.
\begin{proposition} \label{pro:c1}
We have the following:\\
$(1)$ 
Let $1/6 < a < 1/4$ and $0 < \sigma < 1$. Then we have
$$
Z(\sigma,a) < 0 \quad \mbox{and} \quad P(\sigma,a) < 0.
$$
$(2)$ When $0 < a < 1/6$, the functions $Z(\sigma,a)$ and $P(\sigma,a)$ have precisely one simple zero in the open interval $(0, 1)$. \\
$(3)$ For $0 < a < 1/6$, let $\beta_Z(a)$ and $\beta_P(a)$ denote the unique zero of $Z(\sigma,a)$ and $P(\sigma,a)$ in $(0,1)$, respectively. Then the function $\beta_Z(a) \, \colon (0, 1/6) \to (0, 1)$ is a strictly decreasing $C^\infty$-diffeomorphism and $\beta_P(a) \, \colon (0, 1/6) \to (0, 1)$ is a strictly increasing $C^\infty$-diffeomorphism. Furthermore, as $a \to +0$, it holds that
\begin{equation}\label{eq:as1}
\beta_Z(a) = 1 - 2a + 2a^2 \log a + O(a^2), \qquad \beta_P(a) = 2a - 2a^2 \log a + O(a^2). 
\end{equation}
\end{proposition}

\begin{proof}
First we show (1) for $P(\sigma,a)$. Let $1/6 < a < 1/4$. Then we have
\begin{equation*}
\begin{split}
&\lim_{\sigma \to +0}\bigl(-\log (\cos 2 \pi a) \bigr)^{-\sigma} \Gamma (\sigma) P(\sigma,a) = -\infty, \\
&\lim_{\sigma \to 1-0}\bigl(-\log (\cos 2 \pi a) \bigr)^{-\sigma} \Gamma (\sigma) P(\sigma,a) = 
\frac{2 \log (2 \sin \pi a)}{\log (\cos 2 \pi a)} <0
\end{split}
\end{equation*}
from (\ref{eq:pv1}). Hence we have $P(\sigma,a) < 0$ for all $0 < \sigma < 1$ by Lemma \ref{lem:3}. Recall that 
\begin{equation}\label{ine:GC}
\Gamma (\sigma)>0 \quad \mbox{and} \quad \cos ( \sigma /2) > 0, \qquad 0 < \sigma < 1. 
\end{equation}
Hence we have (1) for $Z(\sigma,a)$ by the functional equation (\ref{eq:zfe1}). 

Next assume $0 < a < 1/6$. Then, from (\ref{eq:pv1}), we have
$$
\lim_{\sigma \to +0} \alpha^{-\sigma} \Gamma (\sigma) P(\sigma,a) = -\infty, \qquad
\lim_{\sigma \to 1-0} \alpha^{-\sigma} \Gamma (\sigma) P(\sigma,a) >0, 
$$
where $\alpha := -\log (\cos 2 \pi a)>0$. Thus the function $P(\sigma,a)$ has precisely one simple zero $0<\beta_P(a) <1$ by Lemma \ref{lem:3}. From the functional equation (\ref{eq:zfe1}), we have
$$
Z (1-\beta_P(a),a) = 
\frac{2\Gamma (\beta_P(a))}{(2\pi)^{\beta_P(a)}}  \cos \Bigl( \frac{\pi \beta_P(a)}{2} \Bigr) P(\beta_P(a),a)
=0.
$$
Therefore, by (\ref{ine:GC}), it holds that
$$
\beta_Z(a) = 1-\beta_P(a).
$$
The equality implies (2) for $Z(\sigma,a)$.

Recall that $Z(\sigma, a)$ with $0 < \sigma < 1$ is strictly decreasing with respect to $0 < a \le 1/2$ from Lemma \ref{lem:4}. Assume $0 < a_1 < a_2 < 1/6$. Then one has
$$
0 = Z (\beta_Z(a_1),a_1) > Z (\beta_Z(a_1),a_2) .
$$
From the uniqueness of the zero of $Z(\sigma, a)$, it holds that
$$
Z(\sigma, a_2)<0 \quad \mbox{if and only if} \quad 0 < \beta_Z(a_2) < \sigma <1.
$$
Thus we have $\beta_Z(a_1) > \beta_Z(a_2)$. Therefore, we have the monotonicity of $\beta_Z$ which implies that $\beta_Z$ is injective. Fix $\sigma  \in (0,1)$. Then we have
$$
\lim_{a \to +0} Z(\sigma, a) = \lim_{a \to +0} \Bigl( a^{-\sigma} + \zeta (\sigma,1+a) + \zeta (\sigma,1-a) \Bigr)
= \infty. 
$$
On the other hand, from (\ref{eq:z1/6}) and \cite[(2.12.4)]{Tit}, we have
$$
Z(\sigma, 1/6) = (2^\sigma-1) (3^\sigma-1) \zeta (\sigma) <0.
$$
Hence, there exists $0<a<1/6$ such that $Z(\sigma, a)=0$. Therefore, $\beta_Z$ is surjective. By (\ref{eq:parDa}) and the holomorphy of $\zeta (s,a)$, two variable function 
$$
Z (\cdot, \cdot) \, \colon (0,1) \times (0,1/6) \to {\mathbb{R}}
$$
is $C^\infty$. Hence the function $\beta_Z(a)$ is $C^\infty$ from $(\partial / \partial a) Z (\sigma,a) < 0$, which is proved in Lemma \ref{lem:4}, the implicit function theorem and (2) of this proposition. Similarly, we can see that the inverse of $\beta_Z(a)$ is also $C^\infty$ from the inverse function theorem. By using $\beta_P(a) = 1-\beta_Z(a)$, $(\partial / \partial a) P (\sigma,a) < 0$ and modifying the argument above, we can show that $\beta_P(a)$ is strictly increasing $C^\infty$-diffeomorphism.

When $\sigma >1$ and $0 < a \le 1/2$, it holds that
$$
\zeta (s, 1+a) - \zeta (s) = \sum_{n=1}^\infty \frac{(1+a/n)^{-s} -1}{n^s} =
\sum_{n=1}^\infty n^{-s} \biggl( -s \frac{a}{n} + \frac{s (s+1)}{2} \frac{a^2}{n^2} - \cdots \biggr).
$$
The series above converges absolutely if $\sigma >0$. Hence, for $\sigma \ge 1/2$, one has
$$
\zeta (\sigma, 1+a) - \zeta (\sigma) =  O(1), \qquad \zeta (\sigma, 1-a) - \zeta (\sigma) =  O(1).
$$
Therefore, by (\ref{eq:1res}), it holds that
\begin{equation*}
\begin{split}
&Z(\sigma,a) = \zeta (\sigma, a) + \zeta (\sigma, 1-a) = a^{-\sigma}  + \zeta (\sigma, 1+a) + \zeta (\sigma, 1-a)
\\ &= a^{-\sigma}  + 2\zeta(\sigma) + O(1) = a^{-\sigma}  + \frac{2}{\sigma-1} + O(1)
\end{split}
\end{equation*}
when $1-\sigma >0$ and $a >0 $ are sufficiently small. Take $\sigma = \beta : = \beta_Z(a)$. Then we obtain
\begin{equation}\label{eq:es1}
\beta -1 = -2 a^{\beta} + O \bigl( (1-\beta) a^{\beta} \bigr). 
\end{equation}
One has $\beta -1 \ll a ^\beta \ll a^{1/2}$ from the assumption $1/2 \le \sigma < 1$. Hence we have
$$
a^\beta = a \exp \bigl( (\beta-1) \log a \bigr) = 
a + (\beta-1) a \log a + O\bigl( (\beta-1)^2 a |\log a|^2 \bigr) . 
$$
In particular, one has $a^\beta \ll a$. By substituting the estimates above into (\ref{eq:es1}), we obtain
\begin{equation}\label{eq:es2}
\beta -1 = -2 a + 2 (1-\beta) a \log a + O \bigl( (1-\beta) a \bigr). 
\end{equation}
Especially, it holds that
$$
\beta -1 \ll a \quad \mbox{and} \quad \beta -1 = -2a + O\bigl( a^2 |\log a| \bigr) . 
$$
Substituting the estimates above into (\ref{eq:es2}), we get
\begin{equation*}
\begin{split}
\beta -1 &= -2a + \bigl( 2a + O(a^2 |\log a|) \bigr) a \log a + O(a^2) \\ &=
 -2a + 2a^2 \log a + O(a^2).
\end{split}
\end{equation*}
Therefore, we obtain the asymptotic formulas (\ref{eq:as1}) according to $\beta_P(a) = 1-\beta_Z(a)$.
\end{proof}

\subsection{When ${\bm{ \sigma > 1 }}$ or ${\bm{ \sigma < 0 }}$}
In this subsection, we consider the case $\sigma \not \in [0,1]$.
\begin{proposition}\label{pro:a1}
Assume $\sigma >1$. We have the following:\\
$(1)$ 
Let $0 < a \le 1/2$. Then we have $Z(\sigma,a) > 0$ for all $\sigma >1$.\\
$(2)$ Let $0 < a < 1/6$. Then we have $P(\sigma,a) > 0$ for all $\sigma >1$.\\
$(3)$ When $1/6 < a < 1/4$, the functions $P(\sigma,a)$ have precisely one simple zero in $(1, \infty)$. Furthermore, $\beta_P(a) \, \colon (1/6, 1/4) \to (1, \infty)$ is a strictly increasing $C^\infty$-diffeomorphism and it holds that
\begin{equation}\label{eq:asypinf1}
\beta_P(a) = - \frac{\log (\cos 2\pi a)}{\log 2} + O\bigl( \cos 2\pi a \bigr)
\end{equation}
as $a \to 1/4-0$, namely, $\cos 2 \pi a \to +0$.
\end{proposition}
\begin{proof}
The statement (1) is immediately proved by
$$
Z(\sigma,a) = \zeta (\sigma,a) + \zeta (\sigma,1-a) = 
\sum_{n=0}^\infty \Bigl( (n+a)^{-\sigma} + (n+1-a)^{-\sigma}  \Bigr) >0.
$$

Suppose $0 < a < 1/6$. From (\ref{eq:pv1}) and the series expression of $P(s,a)$, we have
$$
P (1,a) = -2 \log (2 \sin \pi a) >0, \qquad \lim_{\sigma \to \infty} P (\sigma,a) = \cos 2\pi a >0.
$$
Hence it holds that 
$$
\alpha^{-1} \Gamma (1) P(1,a) >0, \qquad 
\lim_{\sigma \to \infty} \alpha^{-\sigma} \Gamma (\sigma) P(\sigma,a) = \infty,
$$
where $\alpha := -\log (\cos 2 \pi a)>0$. Recall that the function $\alpha^{-\sigma} \Gamma (\sigma) P(\sigma,a)$ is strictly increasing for $\sigma >0$ from Lemma \ref{lem:3}. Therefore, one has $P(\sigma,a) > 0$ for all $\sigma >1$ if $0 < a < 1/6$ since $\Gamma (\sigma) >0$ when $\sigma >1$.

Suppose $1/6 < a < 1/4$. Then we have
$$
P (1,a) = -2 \log (2 \sin \pi a) <0, \qquad \lim_{\sigma \to \infty} P (\sigma,a) = \cos 2\pi a >0.
$$
Thus, from Stirling's formula for the gamma function, one has
$$
\alpha^{-1} \Gamma (1) P(1,a) <0, \qquad 
\lim_{\sigma \to \infty} \alpha^{-\sigma} \Gamma (\sigma) P(\sigma,a) = \infty.
$$
Hence the function $P(\sigma,a)$ has precisely one simple zero in $(1, \infty)$ by the monotonicity of $\alpha^{-\sigma} \Gamma (\sigma) P(\sigma,a)$ proved in Lemma \ref{lem:3}. By using $(\partial / \partial a) P (\sigma,a) < 0$ proved in Lemma \ref{lem:4} and modifying the proof of (3) of Proposition \ref{pro:c1}, we have that $\beta_P(a)$ is strictly increasing with respect to $1/6 < a < 1/4$. The monotonicity of $\beta_P(a)$ implies that $\beta_P$ is injective. Moreover, we obtain that $\beta_P(a)$ is a $C^\infty$-function, especially, continuous function, from Lemma \ref{lem:4} and the argument appeared in the proof of (3) of Proposition \ref{pro:c1}. On the other hand, it holds that
\begin{equation*}
\begin{split}
&\lim_{a \to 1/6} \lim_{\sigma \to 1} P(\sigma,a) = \lim_{\sigma \to 1} \lim_{a \to 1/6} P(\sigma,a)
= P(1,1/6) = 0, \\
&\lim_{a \to 1/4} \lim_{\sigma \to \infty} P(\sigma,a) = \lim_{\sigma \to \infty} \lim_{a \to 1/4} P(\sigma,a)
= 0 .
\end{split}
\end{equation*}
Therefore, $\beta_P(a)$ is surjective by the equations above, the intermediate value theorem, the continuity and monotonicity of $\beta_P(a)$. Hence $\beta_P(a)$ is a $C^\infty$-diffeomorphism. The inverse of $\beta_P(a)$ is also a $C^\infty$-function by Lemma \ref{lem:4} and the inverse function theorem.

Now let $a-1/4 > 0$ be sufficiently small and $\sigma >1$ be sufficiently large. For any $m \in {\mathbb{N}}$, one has
\begin{equation}\label{eq:pfappz1}
\Biggl| \sum_{n=m}^\infty \frac{\cos 2 \pi na}{n^\sigma} \Biggr| \le \sum_{n=m}^\infty \frac{1}{n^\sigma}
\le m^{-\sigma} + \int_m^\infty \frac{dx}{x^\sigma} = m^{-\sigma} + \frac{m^{1-\sigma}}{\sigma-1} .
\end{equation}
Put $\beta = \beta_P(a)$. Then we have
$$
0 = P(\beta,a) =
\cos 2\pi a + \frac{\cos 4 \pi a}{2^\beta} + \frac{\cos 6 \pi a}{3^\beta} + \frac{\cos 8 \pi a}{4^\beta}
+  \frac{\cos 10 \pi a}{5^\beta} + \cdots .
$$
By applying (\ref{eq:pfappz1}) to terms of the right-hand side of the formula above except for the first and second terms, we have
$$
0 = \cos 2 \pi a + 2^{-\beta} \cos 4\pi a + O(3^{-\beta}).
$$
Hence, from $3^{-\beta} \ll 2^{-\beta}$, one has
\begin{equation}\label{eq:pfappz2}
2^{-\beta} = O(\cos 2 \pi a) .
\end{equation}
According to the triple-angle formula, we obtain
$$
\frac{2^\beta}{3^\beta} \cos 6 \pi a = 
\frac{2^\beta}{3^\beta} \bigl( 4\cos^3 2 \pi a - 3 \cos 2\pi a \bigl) = O \bigl( \cos 2\pi a \bigr).
$$
By using (\ref{eq:pfappz1}) and (\ref{eq:pfappz2}), we get
$$
2^\beta \Biggl| \sum_{n=4}^\infty \frac{\cos 2 \pi na}{n^\beta} \Biggr| \le 2^{-\beta} + 4\frac{2^{-\beta}}{\beta -1}
= O \bigl( \cos 2\pi a \bigr).
$$
Therefore, it holds that
$$
0 = 2^\beta \cos 2\pi a + \cos 4 \pi a + 2^\beta \sum_{n=3}^\infty \frac{\cos 2 \pi na}{n^\beta} 
= 2^\beta \cos 2\pi a + \cos 4 \pi a + O \bigl( \cos 2\pi a \bigr).
$$
Hence, by the equation above, one has
\begin{equation*}
\begin{split}
&\beta \log 2 = \log \Bigl( \frac{-\cos 4\pi a + O ( \cos 2\pi a )}{\cos 2\pi a} \Bigr)
= - \log (\cos 2\pi a ) + \log \bigl( - \cos 4\pi a + O ( \cos 2\pi a ) \bigr)\\
&= - \log (\cos 2\pi a ) + \log \bigl( 1 - 2 \cos^2 2\pi a + O ( \cos 2\pi a ) \bigr)
= - \log (\cos 2\pi a ) +  O \bigl( \cos 2\pi a \bigr).
\end{split}
\end{equation*}
The formula above implies the asymptotic formula (\ref{eq:asypinf1}). 
\end{proof}

\begin{proposition} \label{pro:a2}
Suppose $\sigma <0$. We have the following:\\
$(1)$ 
Let $0 < a < 1/4$. Then all real zeros of $P(s,a)$ are simple and only at the negative even integers.\\
$(2)$ Let $0 < a < 1/6$. Then all real zeros of $Z(s,a)$ are simple and only at the negative even integers.\\
$(3)$ When $1/6 < a < 1/4$, we have $Z(-2n,a) = 0$, $n \in {\mathbb{N}}$ and $Z(1-\beta_P(a),a) =0$, where $\beta_P(a) \, \colon (1/6, 1/4) \to (1, \infty)$ is appeared in Proposition \ref{pro:a1}. Moreover, we have
\begin{equation}\label{eq:asypinf2}
\beta_Z(a) = 1- \beta_P(a) =  \frac{\log (\cos 2\pi a)}{\log 2} + 1 + O\bigl( \cos 2\pi a \bigr)
\end{equation}
as $a \to 1/4-0$.
\end{proposition}
\begin{proof}
The statements above are easily proved by Proposition \ref{pro:a1}, the functional equations (\ref{eq:zfe1}) and (\ref{eq:zfe1p}). I should be noted that for any fixed $l \in {\mathbb{N}}$, there exists precisely one $1/6 < a_l < 1/4$ such that 
$$
\beta_Z(a_l) := 1-\beta_P(a_l) = -2l
$$
from bijectivity of $\beta_P(a) \, \colon (1/6, 1/4) \to (1, \infty)$. In this case, the all real zeros of $Z(s,a)$ are only at the negative even integers. However, the real zero at $s=-2l$ is not simple but double. 
\end{proof}

\subsection{Proofs of Theorems \ref{th:m1} and \ref{th:mt1}}
By using (\ref{eq:z1/6}) and (\ref{eq:p1/6}), we have the following.
\begin{proposition}\label{cor:Z}
When $0 < a < 1/4$, the function $Z(\sigma,a)$ have zeros at the non-positive integers and real zero $\beta_Z(a)$ in $(-\infty,1)$. 
\end{proposition}
\begin{proof}
By (\ref{eq:z1/6}), the function $Z(s,a)$ has a double real zero at $\sigma =0$ when $a=1/6$. Hence we have this proposition from (3) of Propositions \ref{pro:c1} and \ref{pro:a2}. 
\end{proof}

\begin{proposition}\label{cor:p}
When $0 < a < 1/4$, the function $P(\sigma,a)$ have precisely one simple zero $\beta_P(a)$ in $(0, \infty)$. Furthermore, the function $\beta_P(a) \, \colon (0, 1/4) \to (0,\infty)$ is a strictly increasing $C^\infty$-diffeomorphism.
\end{proposition}
\begin{proof}
From (\ref{eq:p1/6}), the function $P(s,a)$ has a simple real zero at $\sigma =1$ when $a=1/6$. Thus $\beta_P(a) \, \colon (0, 1/4) \to (0, \infty)$ is a bijection from (3) of Propositions \ref{pro:c1} and \ref{pro:a1}. We can see that $\beta_P(a)$ is a strictly increasing $C^\infty$-function by $(\partial / \partial a) P(\sigma,a) < 0$ proved in Lemma \ref{lem:4} and the method used in the proof of (3) of Proposition \ref{pro:c1}. We can show that the inverse of $\beta_P(a)$ is also a $C^\infty$-function likewise.
\end{proof}

In order to prove Theorems \ref{th:m1} and \ref{th:mt1}, we show the following.
\begin{proposition}\label{pro:ksz1}
Let $1/4 \le a \le 1/2$. Then all real zeros of $Z(s,a)$ are simple and only at the non-positive even integers.
\end{proposition}
\begin{proof}
One has $Z(\sigma,a) >0$ when $\sigma >1$ from (1) of Proposition \ref{pro:a1}. 
When $1/4 \le a \le 1/2$, one has $-1 \le \cos 2\pi a \le 0$. Hence for all $x>0$, we have
$$
\Re \biggl( \frac{e^{2\pi ia}}{e^x-e^{2\pi ia}} \biggr)  =
\frac{e^x \cos 2\pi a -1}{(e^x - \cos 2\pi a)^2 + \sin ^2 2\pi a} <0 
$$
by (\ref{eq:PImk1}). The inequality above and (\ref{eq:intzc1}) imply
\begin{equation}\label{eq:kspneg1}
P(\sigma,a) = 2\sum_{n=1}^\infty \frac{\cos 2\pi na}{n^\sigma} =
\Re \biggl( \frac{e^{2\pi ia}}{\Gamma (\sigma)} \int_0^\infty \frac{2 x^{\sigma-1}}{e^x-e^{2\pi ia}} dx \biggr)< 0,
\qquad \sigma >0 .
\end{equation}

It is widely known that $\Gamma (s)$ is analytic when $\sigma >0$ and 
\begin{equation}\label{eq:gamposi}
\Gamma (\sigma) > 0, \qquad \sigma >0.
\end{equation}
Therefore, all real zeros of $Z(1-s,a)$ with $1/4 \le a \le 1/2$ and $\sigma >0$ come from $\cos ( \pi s/2 ) = 0$ with $\sigma >0$ by (\ref{eq:zfe1}), (\ref{eq:kspneg1}) and (\ref{eq:gamposi}). Hence, every real zero of $Z(s,a)$ with $1/4 \le a \le 1/2$ and $\sigma <1$ is caused by $\cos ( \pi (1-s)/2) = 0$ with $\sigma < 1$ which is equivalent to that $s$ is a non-positive even integer. Furthermore, we can easily see that all real zeros of $Z(s,a)$ are simple by the equation above. 
\end{proof}

\begin{proposition}\label{pro:ksp1}
Assume $1/4 \le a \le 1/2$. Then all real zeros of the function $P(s,a)$ are simple and only at the negative even integers.
\end{proposition}
\begin{proof}
We have $P(\sigma ,a) <0$ when $\sigma >0$ and $1/4 \le a \le 1/2$ by (\ref{eq:kspneg1}). Moreover, we have $P(0,a)=-1$ from (\ref{eq:pv1}). Therefore, all real zeros of $P(1-s,a)$ with $1/4 \le a \le 1/2$ and $\sigma > 1$ come from $\cos (\pi s/2) = 0$ with $\sigma >1$ by (\ref{eq:zfe1p}), (\ref{eq:gamposi}) and (1) of Proposition \ref{pro:a1}. Obviously, all real zeros of the function $P(s,a)$ are simple by the equation above or the functional equation (\ref{eq:zfe1p}). 
\end{proof}

Now we are in position to prove Theorems \ref{th:m1} and \ref{th:mt1}
\begin{proof}[Proof of Theorem \ref{th:m1}]
When $0 < a < 1/6$, $Z(s,a)$ has precisely one simple real zero in the open interval $(0, 1)$ from (3) of Proposition \ref{pro:c1}. If $a=1/6$, $Z(s,a)$ has a double real zero at $\sigma =0$ by (\ref{eq:z1/6}). Let $1/6 < a < 1/4$. Then, from (3) of Proposition \ref{pro:a2}, $Z(s,a)$ has a double real zero at $\sigma = -2n$, $n \in {\mathbb{N}}$ or a simple real zero for $\sigma < 0$ and $\sigma \ne -2n$. Proposition \ref{pro:ksz1} implies the case $1/4 \le a \le 1/2$.
\end{proof}

\begin{proof}[Proof of Theorem \ref{th:mt1}]
From Proposition \ref{cor:p}, the functions $P(\sigma,a)$ have precisely one simple zero in $(0, \infty)$ when $0 < a < 1/4$. Proposition \ref{pro:ksp1} implies the case $1/4 \le a \le 1/2$.
\end{proof}

\subsection{Proofs of Propositions \ref{pro:z1}, \ref{pro:p1} and \ref{pro:zero1}}
We prove the remaining propositions.
\begin{proof}[Proofs of Propositions \ref{pro:z1} and \ref{pro:p1}]
These are easily proved by Propositions \ref{pro:c1}, \ref{pro:a1} and \ref{pro:a2}, \ref{cor:Z} and \ref{cor:p}.
\end{proof}

We prove Proposition \ref{pro:zero1} only for the function $Z(s,a)$. We can show this proposition for other zeta functions similarly. 
\begin{proof}[Proof of Proposition \ref{pro:zero1}]
The upper bound for the number of zeros  of $Z(s,a)$ is proved by the Bohr-Landau method (see \cite[Theorem 9.15 (A)]{Tit}), the mean square of $\zeta (s,a)$ (see \cite[Theorem 4.2.1]{LauGa}) and the inequality
$$
\int_2^T \bigl| Z(\sigma+it,a) \bigr|^2 dt \le 
2\int_2^T \bigl| \zeta(\sigma+it,a) \bigr|^2 dt + 2\int_2^T \bigl| \zeta(\sigma+it,1-a) \bigr|^2 dt,
\qquad \sigma >1/2.
$$

First assume that $a \ne 1/2,1/3,1/4,1/6$ is rational. Then the lower bounds for the number of zeros of $Z(s,a)$ with $\sigma >1$ and $1/2 < \sigma <1$ are prove by (\ref{eq:qq}), \cite[Corollary]{SW}, \cite[Theorem 2]{KaczorowskiKulas}, the definition of the zeta function $Z(s,a)$ and the fact that $\varphi (q) \le 2$ if and only if $q=1,2,3,4,6$.

Next suppose $a$ is transcendental and $1/2 < \sigma <1$. Then $\zeta (s,a)$ and $\zeta(s,1-a)$ have the joint universality by \cite[Theorem 5]{MZ}. Hence we can prove Proposition \ref{pro:p1} in this case by modifying the proof of \cite[Theorem 2]{KaczorowskiKulas} (see also \cite[Theorem 3]{KaczorowskiKulas}). 

Finally, consider the case $a$ is transcendental and $\sigma >1$. We can easily see that the set
$$
\bigl\{ \log (n+a) : n \in {\mathbb{N}} \cup \{0\} \bigr\} \cup 
\bigl\{ \log (n+1-a) : n \in {\mathbb{N}} \cup \{0\} \bigr\}
$$
is linearly independent over ${\mathbb{Q}}$. Hence, by using the Davenport and Heilbronn method (see for example \cite[p.~162]{LauGa}), the function $Z(s,a)$ has more than $C_a^\flat (T)$ in the rectangle $1 < \sigma < 1+\delta$ and $0 < t < T$ when $T$ is sufficiently large. 
\end{proof}

\subsection*{Acknowledgments}
The author was partially supported by JSPS grant 16K05077. \\

 

\begin{thebibliography}{1}
\bibitem{Apo} 
T.~M.~Apostol, \textit{Introduction to Analytic Number Theory}. Undergraduate Texts in Mathematics, Springer, New York, 1976.

\bibitem{Cohen} 
{\rm H.~Cohen}, {\em Number theory. Vol.~II. Analytic and modern tools} (Graduate Texts in Mathematics, 240. Springer, New York, 2007). 

\bibitem{ConSou} 
J.~B.~Conrey and K.~Soundararajan, {\it{Real zeros of quadratic Dirichlet $L$-functions}}. Invent.~Math. {\bf{150}} (2002), no.~1, 1--44. 

\bibitem{ES} 
K.~Endo and Y.~Suzuki, {\it{Real zeros of Hurwitz zeta-functions and their asymptotic behavior in the interval $(0,1)  $}}, J.~Math.~Anal.~Appl. {\bf{473}} (2019), no.~2, 624--635.

\bibitem{Fine}
N. J. Fine, {\it{Note on the Hurwitz zeta-function}}, Proceedings of the American Mathematical Society. {\bf{2}}, (1951), 361--364.

\bibitem{KaczorowskiKulas} J.~Kaczorowski and M.~Kulas, {\it{On the non-trivial zeros off line for $L$-functions from extended Selberg class}}, Monatshefte Math. \textbf{150} (2007), no.~3, 217--232.

\bibitem{KaVo} 
A.~A.~Karatsuba and S.~M.~Voronin, {\it{The Riemann zeta-function}}. Translated from the Russian by Neal Koblitz. De Gruyter Expositions in Mathematics, 5. Walter de Gruyter $\&$ Co., Berlin, 1992.

\bibitem{Kob} 
N.~Koblitz, {\it{Introduction to elliptic curves and modular forms}}. Second edition. Graduate Texts in Mathematics, {\bf{97}}. Springer-Verlag, New York, 1993. 

\bibitem{LauGa}
A.~Laurin\v{c}ikas and R.~Garunk\v{s}tis, {\it{The Lerch zeta-function}}. Kluwer Academic Publishers, Dordrecht, 2002.

\bibitem{Mts}
T.~Matsusaka, {\it{Real zeros of the Hurwitz zeta function}}, Acta Arithmetica, {\bf{183}} (2018),  no.~1, 53--62. (arXiv: 1610.07945).

\bibitem{MZ}
H.~Mishou, {\it{The joint universality theorem for a pair of Hurwitz zeta functions}}. J.~Number Theory {\bf{131}} (2011), no.~12, 2352--2367. 


\bibitem{NaC}
T.~Nakamura, {\it{Real zeros of Hurwitz-Lerch zeta and Hurwitz-Lerch type of Euler-Zagier double zeta functions}}. Math.~Proc.~Cambridge Philos.~Soc. {\bf{160}} (2016), no.~1, 39--50. 

\bibitem{NJMAA}
T.~Nakamura, {\it{Real zeros of Hurwitz-Lerch zeta functions in the interval $(-1,0)$}}. J.~Math.~Anal.~Appl. {\bf{438}} (2016),  no.~1, 42--52.

\bibitem{NPFE}
T.~Nakamura, {\it{Functional equation and zeros on the critical line of the quadrilateral zeta function}}, to appear in J.~Number Theory, arXiv:1910.09837.

\bibitem{NaIn}
T.~Nakamura, {\it{The values of zeta functions composed by the Hurwitz and periodic zeta functions at integers}}, arXiv:2006.03300
.

\bibitem{SW}
E.~Saias and A.~Weingartner, {\it{Zeros of Dirichlet series with periodic coefficients}}, Acta Arith. {\bf{140}} (2009), no.~4, 335--344. 

\bibitem{Tit} E.~C.~Titchmarsh, {\it{The theory of the Riemann zeta-function,}} Second edition. Edited and with a preface by D.~R.~Heath-Brown. The Clarendon Press, Oxford University Press, New York, 1986. 

\end{thebibliography}
\end{document}